\documentclass[preprint,12pt]{elsarticle}
\usepackage[nodots,nocompress]{numcompress}
\usepackage{epsfig,amsmath,amssymb,latexsym,amsfonts,curves,eepic,epsfig,epic,graphics,amsthm,graphicx}
\usepackage{hyperref}
\numberwithin{equation}{section}
\newtheorem{theorem}{Theorem}[section]

\newtheorem{remark}{Remark}[section]

\usepackage{mathrsfs}
\usepackage{mathtools}
\usepackage{graphicx}
\usepackage{stmaryrd}
 \usepackage{subfigure}
\usepackage{float}
\usepackage{amsmath}
    \usepackage{bm}
\usepackage{amsfonts,amssymb}
   \usepackage{natbib}
  \usepackage{dsfont}
\usepackage{amsfonts}
\usepackage{mathrsfs}
  \usepackage{pifont}
\usepackage{xcolor}
\usepackage{booktabs}

\begin{document}
\begin{frontmatter}
\title{ {\Large A novel directly energy-preserving method for charged particle dynamics}}
\author{Yexin Li$^{a,b}$, Ping Jiang$^{c}$, Haochen Li$^{a,b,*}$ }
\address{
a School of Science, Beijing University of Posts and Telecommunications,
Beijing 100876, China \\
b Key Laboratory of Mathematics and Information Networks(Beijing University of Posts and Telecommunications),
Ministry of Education, China \\
c Peking University Chongqing Research Institute of Big Data, Chongqing 400031, China
}
\begin{abstract}

In this paper, we apply the coordinate increment discrete gradient (CIDG) method to solve the Lorentz force system
which can be written as a non-canonical Hamiltonian system. Then we can obtain a new energy-preserving CIDG-I method for the system.
The CIDG-I method can combine with its adjoint method CIDG-II which is also a energy-preserving method to form a new  method, namely CIDG-C method.
The CIDG-C method is symmetrical and can conserve
the Hamiltonian energy directly and exactly.
With comparison to the well-used Boris method, numerical experiments indicate that
the CIDG-C method holds advantage over the Boris method in terms of energy-conserving.

\textcolor{red}{The Fig. 5(b) in the original paper~\cite{Ref2024Li} contains an error. We submit the correct Fig. \ref{fig:5}(b) and an errata in this paper, which is described in remark \ref{remark:3}.}
\end{abstract}
\end{frontmatter}

\begin{figure}[b]
\small \baselineskip=10pt
\rule[2mm]{1.8cm}{0.2mm} \par
$^{*}$Corresponding author.\\
E-mail address: lihaochen\_bjut@sina.com.
\end{figure}

\pagestyle{myheadings}
\markboth{\hfil 
   \hfil \hbox{}}
{\hbox{} \hfil 
Yexin Li, Ping Jiang, Haochen Li  \hfil}

\section{Introduction}
\label{intro}

With the continuous breakthrough the theory of scientific computing and the explosive growth of its application across diverse scientific discipline, the numerical solution of the Hamilton system has gradually become the focus of attention of mathematicians and physicists. For a system, a numerical method which can conserve one or more geometric properties is called geometric numerical integration method, or structure-preserving method~\cite{Ref1982Dodd,Ref1986Feng,Ref2010Feng,Ref2016Brugnano}, such as symplectic methods, volume-preserving methods, energy-preserving methods, and so forth.
Geometric numerical integration method has become an essential tool to solve the Hamiltonian systems. And the method has been effectively applied in many equations, including but not limited to H\'{e}non-Heiles system, nonlinear Schr\"{o}dinger equation, Maxwell's equation.~\cite{Ref2001Chen,Ref2006Hairer,Ref1993Hairer,Ref2010Kong,Ref2008Wang,Ref2014Xu}.

Because the collisions between the charged particles occur infrequently in hot plasmas, many important phenomenons in hot plasmas can be simplified and considered as a charged particle motion.
Therefore, the motion of a charged particle under the influence of electromagnetic fields is one of the most fundamental process in hot plasmas~\cite{Ref2008Bellan,Ref2013Qin}.
Long-term numerical calculation on the trajectories of a charged particle have been widely used to study its dynamical behaviors. However, the non-geometric method such as the standard 4th order Runge-Kutta method is not accurate in the long term simulation.
For a charged particle in the electromagnetic field, its dynamics is governed by Lorentz force system, which can be written as a non-canonical Hamiltonian system~\cite{Ref1981Littlejohn,Ref1980Morrison}.
Therefore, the solution of the Lorentz force system can be derived by applying the geometric numerical integration method on its non-canonical Hamiltonian system.

A considerable amount of literature on geometric numerical integration method has been published, including the volume-preserving algorithm~\cite{Ref2015He,Ref2015Zhang}, the Boris method which can conserve phase space volume~\cite{Ref2013Qin,Ref2005Birdsall,Ref1970Boris,Ref2002Stoltz}, the variational symplectic method~\cite{Ref1983Littlejohn,Ref2008Qin, Ref2009Qin}.
The Boris method, which is the $de$ $facto$ standard for advancing a charged
particle in an electromagnetic field and conserves phase space volume exactly.
Albeit the Boris algorithm has proven itself highly effective for charged particle dynamics in an electromagnetic field, but it does not conserve energy and a numerical drift may be observed inevitably~\cite{Ref2018Hairer}.
It is well-known that the most essential feature of a Hamiltonian system is the
Hamiltonian energy. So We hope to construct an energy-preserving or approximate energy-preserving algorithm to solve the Lorentz force system numerically, and obtain long-term correct results without numerical drift.

Based on numerical quadrature formula, the Bool discrete line integral method ~\cite{Ref2016Li} and the high-order energy-conserving line integral methods~\cite{Ref2019Brugnano} can
conserve the energy exactly for the polynomial Hamiltonian, and conserve the energy of the non-polynomial Hamiltonian up to the round-off error. The invariant energy quadratization method~\cite{Ref2019Li} can conserve the modified energy, which is the approximation of the original Hamiltonian energy.
However, there are few numerical methods that can directly and exactly conserve the original Hamiltonian energy. If we want to construct an accurate energy-preserving algorithm for the Lorentz force system, we need to focus on the algorithms that preserve the energy of the Hamiltonian system.

Over the past few decades, a number of methods have been proposed to preserve the energy of the Hamiltonian system.
In 1996, one of the most famous energy-preserving methods which can construct integral preserving methods for ordinary differential equations was developed by O. Gonzalez, and the method is known as the discrete gradient method~\cite{Ref1996Gonzalez}.
Subsequently, the discrete variational method
for nonlinear wave equation was proposed by Matsuo~\cite{Ref2003Matsuo}. E.Faou proposed that there is an energy-preserving B-series method in the Hamiltonian system. Moreover,
Quispel and McLaren derived the averaged vector field method~\cite{Ref2008Quispel}. For the Lorentz force system in the form of the Hamiltonian system, this study presents a novel directly energy-preserving method.

Fistly, the Lorentz force system can be write as a non-canonical Hamiltonian form. Then we apply the coordinate increment discrete gradient method~\cite{Ref1988Itoh} to
the Hamiltonian system to get a new energy-preserving method. The new method and its adjoint method are both energy-preserving methods.
From this two methods we can derive a new coordinate increment discrete gradient composition (CIDG-C) method.
A key advantage of the CIDG-C method is that it is symmetrical, while simultaneously conserving the Hamiltonian energy exactly without numerical quadrature formula.

This paper is organized as follows. In Section.~\ref{sec:3}, We introduce the dynamics of charged particles in the electromagnetic field, which is written as a non-canonical Hamiltonian system. By using the coordinate increment discrete gradient method to solve the Hamiltonian
system in Section.~\ref{sec:4}, we obtain a new energy-preserving method and its adjoint method. Then the two methods are combined into a new method, namely CIDG-C method. Numerical experiments are presented in
Section.~\ref{sec:5} to compare
the theoretical results of the new CIDG-C method with the
 Boris method~\cite{Ref1970Boris}. Section.~\ref{sec:6} concludes this paper.

\section{Non-canonical Hamiltonian form of the system}
\label{sec:3}

In this section, we firstly review the Non-canonical Hamiltonian form of the Lorentz force system~\cite{Ref1981Littlejohn,Ref1980Morrison,Ref2015He}.
For a charged particle with position $\textbf{x}(t)=(x(t),y(t),z(t))^{T} \in \mathbb{R}^{3}$ in an electromagnetic field,
its dynamics is described by the following Newton-Lorentz equation
\begin{equation}\label{eq:3s1}
\frac{d^{2}\textbf{x}}{d t^{2}}=\frac{d\textbf{x}}{d t}\times \textbf{B}(\textbf{x})- \nabla U(\textbf{x}),
\end{equation}
where $- \nabla U(\textbf{x}):= \textbf{E}(\textbf{x})$ is an electric field, and $\textbf{B}(\textbf{x})=(B_{1}(\textbf{x}), B_{2}(\textbf{x}), B_{3}(\textbf{x}))^{T}$ is a magnetic field.

Let $\textbf{v} = (v_{1},v_{2},v_{3})^{T} = \frac{d\textbf{x}}{d t}  $ and
\[
\textbf{S}(\textbf{x})=
\left(
       \begin{array}{ccc}
              0   & B_{3}(\textbf{x}) & -B_{2}(\textbf{x}) \\[0.3cm]
              -B_{3}(\textbf{x})  & 0 & B_{1}(\textbf{x}) \\[0.3cm]
              B_{2}(\textbf{x}) & -B_{1}(\textbf{x}) & 0
             \end{array}
\right) = -\textbf{S}^{T}(\textbf{x}),
\]
in \eqref{eq:3s1}, then we can obtain the following ODEs
\begin{align}\label{eq:3s2}
&\frac{d\textbf{x}}{d t}=\textbf{v},\\ \label{eq:3s3}
&\frac{d\textbf{v}}{d t}=\textbf{S}(\textbf{x})\textbf{v} - \nabla U(\textbf{x}),\\ \label{eq:3s4}
& (\textbf{x}(0),\textbf{v}(0)) = (\textbf{x}_{0},\textbf{v}_{0}) \in \Omega \subseteq \mathbb{R}^{3}\times \mathbb{R}^{3}.
\end{align}
Denote $\textbf{z} =(\textbf{x}^{T}, \textbf{v}^{T})^{T} = (x,y,z,v_{1},v_{2},v_{3})^{T}$, the ODEs \eqref{eq:3s2}-\eqref{eq:3s4}
can be written as a non-canonical Hamiltonian system
\begin{equation}\label{eq:3s5}
\frac{d\textbf{z}}{d t}=K(\textbf{z})\nabla H(\textbf{z}),
\end{equation}
where
\[
K(\textbf{z})=
\left(
       \begin{array}{cc}
               0  & I  \\[0.3cm]
              -I  & \textbf{S}(\textbf{x})
             \end{array}
\right)
\]
is a skew-symmetrical matrix, and
$H(\textbf{z}) =\frac{1}{2}\textbf{v}^{T} \textbf{v} + U(\textbf{x})$
is the Hamiltonian energy of the system, which is conserved by
\begin{equation}\label{eq:3s6}
\frac{d}{dt}H(\textbf{z})= \nabla H^{T}(\textbf{z})\frac{d\textbf{z}}{d t} = \nabla H^{T}(\textbf{z}) K(\textbf{z})\nabla H(\textbf{z}) =0.
\end{equation}

\section{Coordinate increment discrete gradient methods}
\label{sec:4}

Notice that the Hamiltonian energy of the system \eqref{eq:3s5} can be exactly preserved by the coordinate increment discrete gradient method~\cite{Ref1988Itoh}
\begin{equation}\label{eq:4s1}
\frac{\textbf{z}^{n+1}-\textbf{z}^{n}}{h} = K(\frac{\textbf{z}^{n+1}+\textbf{z}^{n}}{2}) \bar{\nabla} H(\textbf{z}^{n+1},\textbf{z}^{n}),
\end{equation}
i.e.,
\begin{equation*}\label{eq:4s1-2}
\frac{\textbf{z}^{n+1}-\textbf{z}^{n}}{h} = K(\frac{\textbf{z}^{n+1}+\textbf{z}^{n}}{2})
\left(
       \begin{array}{c}
              \frac{H(z^{n+1}_{1},z^{n}_{2},z^{n}_{3},z^{n}_{4},z^{n}_{5},z^{n}_{6})-H(z^{n}_{1},z^{n}_{2},z^{n}_{3},z^{n}_{4},z^{n}_{5},z^{n}_{6})}{z^{n+1}_{1}-z^{n}_{1}}   \\
              \frac{H(z^{n+1}_{1},z^{n+1}_{2},z^{n}_{3},z^{n}_{4},z^{n}_{5},z^{n}_{6})-H(z^{n+1}_{1},z^{n}_{2},z^{n}_{3},z^{n}_{4},z^{n}_{5},z^{n}_{6})}{z^{n+1}_{2}-z^{n}_{2}}   \\
              \frac{H(z^{n+1}_{1},z^{n+1}_{2},z^{n+1}_{3},z^{n}_{4},z^{n}_{5},z^{n}_{6})-H(z^{n+1}_{1},z^{n+1}_{2},z^{n}_{3},z^{n}_{4},z^{n}_{5},z^{n}_{6})}{z^{n+1}_{3}-z^{n}_{3}}  \\
              \frac{H(z^{n+1}_{1},z^{n+1}_{2},z^{n+1}_{3},z^{n+1}_{4},z^{n}_{5},z^{n}_{6})-H(z^{n+1}_{1},z^{n+1}_{2},z^{n+1}_{3},z^{n}_{4},z^{n}_{5},z^{n}_{6})}{z^{n+1}_{4}-z^{n}_{4}}  \\
              \frac{H(z^{n+1}_{1},z^{n+1}_{2},z^{n+1}_{3},z^{n+1}_{4},z^{n+1}_{5},z^{n}_{6})-H(z^{n+1}_{1},z^{n+1}_{2},z^{n+1}_{3},z^{n+1}_{4},z^{n}_{5},z^{n}_{6})}{z^{n+1}_{5}-z^{n}_{5}}  \\
              \frac{H(z^{n+1}_{1},z^{n+1}_{2},z^{n+1}_{3},z^{n+1}_{4},z^{n+1}_{5},z^{n+1}_{6})-H(z^{n+1}_{1},z^{n+1}_{2},z^{n+1}_{3},z^{n+1}_{4},z^{n+1}_{5},z^{n}_{6})}{z^{n+1}_{6}-z^{n}_{6}}
             \end{array}
\right),
\end{equation*}
where $h$ is the step size, and we consider the coordinate increment discrete gradient as
\begin{equation}\label{eq:4s2}
\bar{\nabla} H(\bar{\textbf{z}},\textbf{z})=
\left(
       \begin{array}{c}
              \frac{H(\bar{z}_{1},z_{2},z_{3},z_{4},...,z_{m})-H(z_{1},z_{2},z_{3},z_{4},...,z_{m})}{\bar{z}_{1}-z_{1}}   \\
              \frac{H(\bar{z}_{1},\bar{z}_{2},z_{3},z_{4},...,z_{m})-H(\bar{z}_{1},z_{2},z_{3},z_{4},...,z_{m})}{\bar{z}_{2}-z_{2}}   \\
              \frac{H(\bar{z}_{1},\bar{z}_{2},\bar{z}_{3},z_{4},...,z_{m})-H(\bar{z}_{1},\bar{z}_{2},z_{3},z_{4},...,z_{m})}{\bar{z}_{3}-z_{3}}  \\
              \vdots\\
              \frac{H(\bar{z}_{1},\bar{z}_{2},...,\bar{z}_{n-1},\bar{z}_{m})-H(\bar{z}_{1},\bar{z}_{2},...,\bar{z}_{m-1},z_{m})}{\bar{z}_{m}-z_{m}}
             \end{array}
\right),
\end{equation}
with the notation $\bar{\textbf{z}} = \textbf{z}^{n+1}$, $\textbf{z} = \textbf{z}^{n}$ and $m=6$. We denote the method \eqref{eq:4s1} as coordinate increment discrete gradient - I method (CIDG-I).

Then we consider the adjoint method of CIDG-I \eqref{eq:4s1}, which we denote as the coordinate increment discrete gradient - II method (CIDG-II) of system \eqref{eq:3s5}
\begin{equation}\label{eq:4s3}
\frac{\textbf{z}^{n+1}-\textbf{z}^{n}}{h} = K(\frac{\textbf{z}^{n+1}+\textbf{z}^{n}}{2})\bar{\nabla} H(\textbf{z}^{n},\textbf{z}^{n+1}),
\end{equation}
i.e.,
\begin{equation*}\label{eq:4s3-2}
\frac{\textbf{z}^{n+1}-\textbf{z}^{n}}{h} = K(\frac{\textbf{z}^{n+1}+\textbf{z}^{n}}{2})
\left(
       \begin{array}{c}
              \frac{H(z^{n+1}_{1},z^{n+1}_{2},z^{n+1}_{3},z^{n+1}_{4},z^{n+1}_{5},z^{n+1}_{6})-H(z^{n}_{1},z^{n+1}_{2},z^{n+1}_{3},z^{n+1}_{4},z^{n+1}_{5},z^{n+1}_{6})}{z^{n+1}_{1}-z^{n}_{1}}   \\
              \frac{H(z^{n}_{1},z^{n+1}_{2},z^{n+1}_{3},z^{n+1}_{4},z^{n+1}_{5},z^{n+1}_{6})-H(z^{n}_{1},z^{n}_{2},z^{n+1}_{3},z^{n+1}_{4},z^{n+1}_{5},z^{n+1}_{6})}{z^{n+1}_{2}-z^{n}_{2}}   \\
              \frac{H(z^{n}_{1},z^{n}_{2},z^{n+1}_{3},z^{n+1}_{4},z^{n+1}_{5},z^{n+1}_{6})-H(z^{n}_{1},z^{n}_{2},z^{n}_{3},z^{n+1}_{4},z^{n+1}_{5},z^{n+1}_{6})}{z^{n+1}_{3}-z^{n}_{3}}  \\
              \frac{H(z^{n}_{1},z^{n}_{2},z^{n}_{3},z^{n+1}_{4},z^{n+1}_{5},z^{n+1}_{6})-H(z^{n}_{1},z^{n}_{2},z^{n}_{3},z^{n}_{4},z^{n+1}_{5},z^{n+1}_{6})}{z^{n+1}_{4}-z^{n}_{4}}  \\
              \frac{H(z^{n}_{1},z^{n}_{2},z^{n}_{3},z^{n}_{4},z^{n+1}_{5},z^{n+1}_{6})-H(z^{n}_{1},z^{n}_{2},z^{n}_{3},z^{n}_{4},z^{n}_{5},z^{n+1}_{6})}{z^{n+1}_{5}-z^{n}_{5}}  \\
              \frac{H(z^{n}_{1},z^{n}_{2},z^{n}_{3},z^{n}_{4},z^{n}_{5},z^{n+1}_{6})-H(z^{n}_{1},z^{n}_{2},z^{n}_{3},z^{n}_{4},z^{n}_{5},z^{n}_{6})}{z^{n+1}_{6}-z^{n}_{6}}
             \end{array}
\right).
\end{equation*}

Finally the CIDG-I method $\Phi_{h}$ \eqref{eq:4s1} and its adjoint method  CIDG-II $\Phi^{*}_{h}$ \eqref{eq:4s3} can be combined~\cite{Ref2006Hairer} into the following coordinate increment discrete gradient composition (CIDG-C) method
\begin{equation}\label{eq:4s4}
\Psi_{h} = \Phi_{\frac{h}{2}} \circ \Phi^{*}_{\frac{h}{2}},
\end{equation}
i.e.,
\begin{align}\label{eq:4s5}
&\frac{\textbf{z}-\textbf{z}^{n}}{\frac{1}{2}h} = K(\frac{\textbf{z}+\textbf{z}^{n}}{2})\bar{\nabla} H(\textbf{z}^{n},\textbf{z}),
\\[0.3cm]\label{eq:4s6}
&\frac{\textbf{z}^{n+1}-\textbf{z}}{\frac{1}{2}h} = K(\frac{\textbf{z}^{n+1}+\textbf{z}}{2})\bar{\nabla} H(\textbf{z}^{n+1},\textbf{z}),
\end{align}
$n = 0,1,2,...,N.$

\begin{theorem}\label{theorem:4s1}
CIDG-I, CIDG-II and CIDG-C methods can all conserve the Hamiltonian energy of the system \eqref{eq:3s5}. Moreover, the CIDG-C method is symmetrical and has order 2.
\end{theorem}
\begin{proof}
From the CIDG-I method \eqref{eq:4s1}, we obtain
\begin{align*}
0  = & h \bar{\nabla} H(\textbf{z}^{n+1},\textbf{z}^{n})^{T} K(\frac{\textbf{z}^{n+1}+\textbf{z}^{n}}{2})\bar{\nabla} H(\textbf{z}^{n+1},\textbf{z}^{n})
= \bar{\nabla} H(\textbf{z}^{n+1},\textbf{z}^{n})^{T}(\textbf{z}^{n+1}-\textbf{z}^{n})
\\[0.3cm]
 = & \frac{H(z^{n+1}_{1},z^{n}_{2},z^{n}_{3},z^{n}_{4},z^{n}_{5},z^{n}_{6})-H(z^{n}_{1},z^{n}_{2},z^{n}_{3},z^{n}_{4},z^{n}_{5},z^{n}_{6})}{z^{n+1}_{1}-z^{n}_{1}} (z^{n+1}_{1}-z^{n}_{1}) +...\\[0.3cm]
 &+
\frac{H(z^{n+1}_{1},z^{n+1}_{2},z^{n+1}_{3},z^{n+1}_{4},z^{n+1}_{5},z^{n+1}_{6})-H(z^{n+1}_{1},z^{n+1}_{2},z^{n+1}_{3},z^{n+1}_{4},z^{n+1}_{5},z^{n}_{6})}{z^{n+1}_{6}-z^{n}_{6}} (z^{n+1}_{6}-z^{n}_{6})\\[0.3cm]
 = & [H(z^{n+1}_{1},z^{n}_{2},z^{n}_{3},z^{n}_{4},z^{n}_{5},z^{n}_{6})-H(z^{n}_{1},z^{n}_{2},z^{n}_{3},z^{n}_{4},z^{n}_{5},z^{n}_{6})]+
[H(z^{n+1}_{1},z^{n+1}_{2},z^{n}_{3},z^{n}_{4},z^{n}_{5},z^{n}_{6})
\\[0.3cm]
& -H(z^{n+1}_{1},z^{n}_{2},z^{n}_{3},z^{n}_{4},z^{n}_{5},z^{n}_{6})]+...+[H(z^{n+1}_{1},z^{n+1}_{2},z^{n+1}_{3},z^{n+1}_{4},z^{n+1}_{5},z^{n+1}_{6})
\\[0.3cm]
&  -H(z^{n+1}_{1},z^{n+1}_{2},z^{n+1}_{3},z^{n+1}_{4},z^{n+1}_{5},z^{n}_{6})]
\\[0.3cm]
 = & H(z^{n+1}_{1},z^{n+1}_{2},z^{n+1}_{3},z^{n+1}_{4},z^{n+1}_{5},z^{n+1}_{6}) - H(z^{n}_{1},z^{n}_{2},z^{n}_{3},z^{n}_{4},z^{n}_{5},z^{n}_{6}),
\end{align*}
i.e., the energy conservation of the method is obtained by
\begin{equation*}
H(\textbf{z}^{n+1}) = H(\textbf{z}^{n}), \quad n=0,1,2,...,N.
\end{equation*}

On the other hand, the energy conservation of the CIDG-II method \eqref{eq:4s3} can also be obtained by
\begin{align*}
0  = & h \bar{\nabla} H(\textbf{z}^{n},\textbf{z}^{n+1})^{T} K(\frac{\textbf{z}^{n}+\textbf{z}^{n+1}}{2})\bar{\nabla} H(\textbf{z}^{n},\textbf{z}^{n+1})
= \bar{\nabla} H(\textbf{z}^{n},\textbf{z}^{n+1})^{T}(\textbf{z}^{n+1}-\textbf{z}^{n})
\\[0.3cm]
 = & \frac{H(z^{n+1}_{1},z^{n+1}_{2},z^{n+1}_{3},z^{n+1}_{4},z^{n+1}_{5},z^{n+1}_{6})-H(z^{n}_{1},z^{n+1}_{2},z^{n+1}_{3},z^{n+1}_{4},z^{n+1}_{5},z^{n+1}_{6})}{z^{n+1}_{1}-z^{n}_{1}} (z^{n+1}_{1}-z^{n}_{1}) +...\\[0.3cm]
 &+
\frac{H(z^{n}_{1},z^{n}_{2},z^{n}_{3},z^{n}_{4},z^{n}_{5},z^{n+1}_{6})-H(z^{n}_{1},z^{n}_{2},z^{n}_{3},z^{n}_{4},z^{n}_{5},z^{n}_{6})}{z^{n+1}_{6}-z^{n}_{6}}
(z^{n+1}_{6}-z^{n}_{6})\\[0.3cm]
 = & H(z^{n+1}_{1},z^{n+1}_{2},z^{n+1}_{3},z^{n+1}_{4},z^{n+1}_{5},z^{n+1}_{6}) - H(z^{n}_{1},z^{n}_{2},z^{n}_{3},z^{n}_{4},z^{n}_{5},z^{n}_{6})
= H(\textbf{z}^{n+1}) - H(\textbf{z}^{n}).
\end{align*}
Because CIDG-I and CIDG-II are both energy-preserving methods, the CIDG-C method  \eqref{eq:4s4} is also obviously energy-preserved.

Finally, take $n+1 = n$, and $h=-h$ in the CIDG-C method \eqref{eq:4s5}-\eqref{eq:4s6}, then we have
\begin{align*}
&\frac{\textbf{z}-\textbf{z}^{n+1}}{-\frac{1}{2}h} = K(\frac{\textbf{z}+\textbf{z}^{n+1}}{2})\bar{\nabla} H(\textbf{z}^{n+1},\textbf{z}),
\\[0.3cm]
&\frac{\textbf{z}^{n}-\textbf{z}}{-\frac{1}{2}h} = K(\frac{\textbf{z}^{n}+\textbf{z}}{2})\bar{\nabla} H(\textbf{z}^{n},\textbf{z}),
\end{align*}
which is just equal to \eqref{eq:4s5}-\eqref{eq:4s6}. Therefore, the CIDG-C method is symmetrical and has order 2~\cite{Ref2006Hairer}.\end{proof}

\begin{remark}\label{remark:1}
In practical calculations, in order to avoid that $z^{n+1}_{i}-z^{n}_{i}$ is used as the denominator and equals to $0$, we take (for instance) $\frac{H(z^{n+1}_{1},z^{n+1}_{2},z^{n+1}_{3},z^{n}_{4},z^{n}_{5},z^{n}_{6})-H(z^{n+1}_{1},z^{n+1}_{2},z^{n}_{3},z^{n}_{4},z^{n}_{5},z^{n}_{6})}{z^{n+1}_{3}-z^{n}_{3}}$ as $\lim\limits_{\epsilon \to 0} \frac{H(z^{n+1}_{1},z^{n+1}_{2},z^{n+1}_{3}+\epsilon,z^{n}_{4},z^{n}_{5},z^{n}_{6})-H(z^{n+1}_{1},z^{n+1}_{2},z^{n}_{3},z^{n}_{4},z^{n}_{5},z^{n}_{6})}{z^{n+1}_{3}+\epsilon-z^{n}_{3}}$, 
\\ which is equal to
$\frac{H(z^{n+1}_{1},z^{n+1}_{2},z^{n+1}_{3},z^{n}_{4},z^{n}_{5},z^{n}_{6})-H(z^{n+1}_{1},z^{n+1}_{2},z^{n}_{3},z^{n}_{4},z^{n}_{5},z^{n}_{6})}{z^{n+1}_{3}-z^{n}_{3}}$ when $(z^{n+1}_{3}-z^{n}_{3}) \neq 0$ 
\\ and equal to $\frac{\partial H}{\partial z_{3}}(z^{n+1}_{1},z^{n+1}_{2},z^{n}_{3},z^{n}_{4},z^{n}_{5},z^{n}_{6})$ when $(z^{n+1}_{3}-z^{n}_{3}) = 0$. Note that since $(z^{n+1}_{3}-z^{n}_{3}) = 0$, the preservation of the Hamiltonian is guaranteed.
\end{remark}

\begin{remark}\label{remark:2}
Applying Eqs. \eqref{eq:4s5}-\eqref{eq:4s6} into the Eqs. \eqref{eq:3s2}-\eqref{eq:3s4}, we can obtain the following equivalent form of the CIDG-C method
\begin{align*}
&\frac{\textbf{x}-\textbf{x}^{n}}{\tau} = \frac{\textbf{v}+\textbf{v}^{n}}{2},
\\
&\frac{\textbf{v}-\textbf{v}^{n}}{\tau} = S(\frac{\textbf{x}+\textbf{x}^{n}}{2})\frac{\textbf{v}+\textbf{v}^{n}}{2}-
\left(
       \begin{array}{c}
              \frac{U(x,y^{n},z^{n})-U(x^{n},y^{n},z^{n})}{x-x^{n}}   \\
              \frac{U(x,y,z^{n})-U(x,y^{n},z^{n})}{y-y^{n}}   \\
              \frac{U(x,y,z)-U(x,y,z^{n})}{z-z^{n}}
             \end{array}
\right),
\\[0.3cm]
&\frac{\textbf{x}^{n+1}-\textbf{x}}{\tau} = \frac{\textbf{v}^{n+1}+\textbf{v}}{2},
\\
&\frac{\textbf{v}^{n+1}-\textbf{v}}{\tau} = S(\frac{\textbf{x}^{n+1}+\textbf{x}}{2})\frac{\textbf{v}^{n+1}+\textbf{v}}{2}-
\left(
       \begin{array}{c}
              \frac{U(x^{n+1},y^{n+1},z^{n+1})-U(x,y^{n+1},z^{n+1})}{x^{n+1}-x}   \\
              \frac{U(x,y^{n+1},z^{n+1})-U(x,y,z^{n+1})}{y^{n+1}-y}   \\
              \frac{U(x,y,z^{n+1})-U(x,y,z)}{z^{n+1}-z}
             \end{array}
\right),
\end{align*}
where $\tau = \frac{h}{2}$, $n = 0,1,2,...,N.$
\end{remark}

\section{Numerical experiments}
\label{sec:5}

In this section, we numerically test the CIDG-C method \eqref{eq:4s5}-\eqref{eq:4s6}and compare the results with the
effective Boris method~\cite{Ref1970Boris}, the Boole discrete line integral (BDLI) method~\cite{Ref2016Li} and the line integral method with parameters (2,2) (which we denote as LIM(2,2))~\cite{Ref2019Brugnano}. All numerical tests have been done on a 2.4 GHz Intel core i5 computer with 16 GB of memory, running Octave 8.2.0.

\subsection{\textbf{2D dynamics in a static electromagnetic field}}
\label{sec:5s1}

Firstly,
we consider the 2D dynamics of the charged particle in a static, non-uniform electromagnetic field which is a significant application in the study of the single particle motion and the guiding center dynamics.
In this case the Lorentz force system is a non-polynomial Hamiltonian system.
Theoretical analysis shows that the analytic orbit of the charged particle is a spiraling circle which have a constant radius.
The large circle corresponds to the $\nabla \cdot \textbf{B}$ drift and the
$\textbf{E}\times \textbf{B}$ drift of the guiding center, and the small circle is a
fast-scale gyromotion~\cite{Ref2013Qin,Ref2015He}.

The static, non-uniform electromagnetic field is taken as
\begin{equation}\label{eq:5s1}
\textbf{B}=\nabla\times \textbf{A}=R\textbf{e}_{z},\quad \textbf{E}=-\nabla U=\frac{10^{-2}}{R^{3}}(x\textbf{e}_{x}+y\textbf{e}_{y}),
\end{equation}
where the potentials are chosen to be $\textbf{A}=\frac{R^{2}}{3}R\textbf{e}_{\xi}$, $U=\frac{10^{-2}}{R}$ in cylindrical coordinates $(R,\xi,z)$ with $R=\sqrt{x^{2}+y^{2}}$. In this example, we take the physical quantities $q=1$, $m=1$, which are normalized by the magnetic field $\textbf{B}$ and the electric field $\textbf{E}$.
The Hamiltonian energy of the system \eqref{eq:3s5}
\begin{equation}\label{eq:5s2}
H =\frac{1}{2}\textbf{v}^{T} \textbf{v} + \frac{0.01}{\sqrt{x^{2}+y^{2}}}
\end{equation}
and the angular momentum
\begin{equation}\label{eq:5s2}
p = R^{2}\frac{d}{dt} \xi + \frac{1}{3}R^{3}
\end{equation}
are constant~\cite{Ref2015He}. As the given electromagnetic field \eqref{eq:5s1} changes slowly with respect to the spatial period of the motion, the magnetic moment $\mu=\frac{\textbf{v}_{\bot}^{2}}{2R}$ is an adiabatic invariant, where $\textbf{v}_{\bot}$ is the component of $\textbf{v}$ perpendicular to $\textbf{B}$.

We consider the initial position $\textbf{x}_{0}=(0, 1, 0)^{T}$ with the initial velocity $\textbf{v}_{0}=(0.1, 0.01, 0)^{T}$,  and the step size is $h =\pi/10$ which is one-twentieth of the characteristic gyro-period $2\pi$.

\begin{figure}[htbp]
  (a)\includegraphics[width=0.35\textwidth]{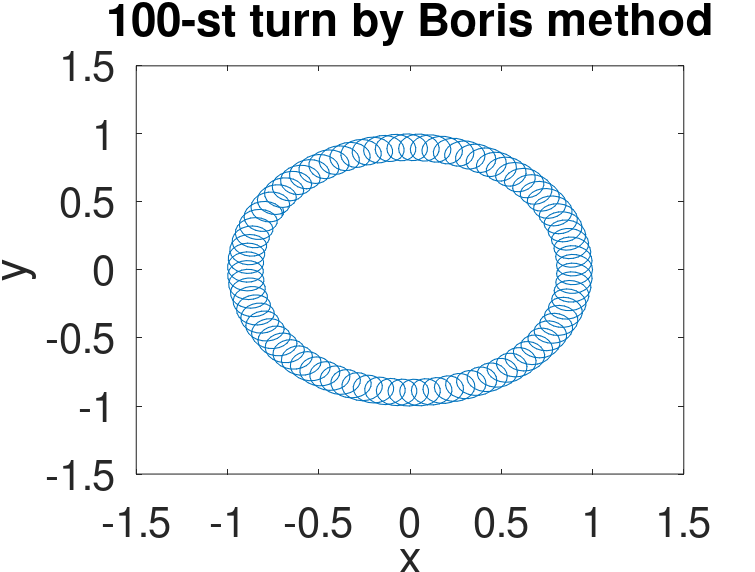}\quad
  (b)\includegraphics[width=0.55\textwidth]{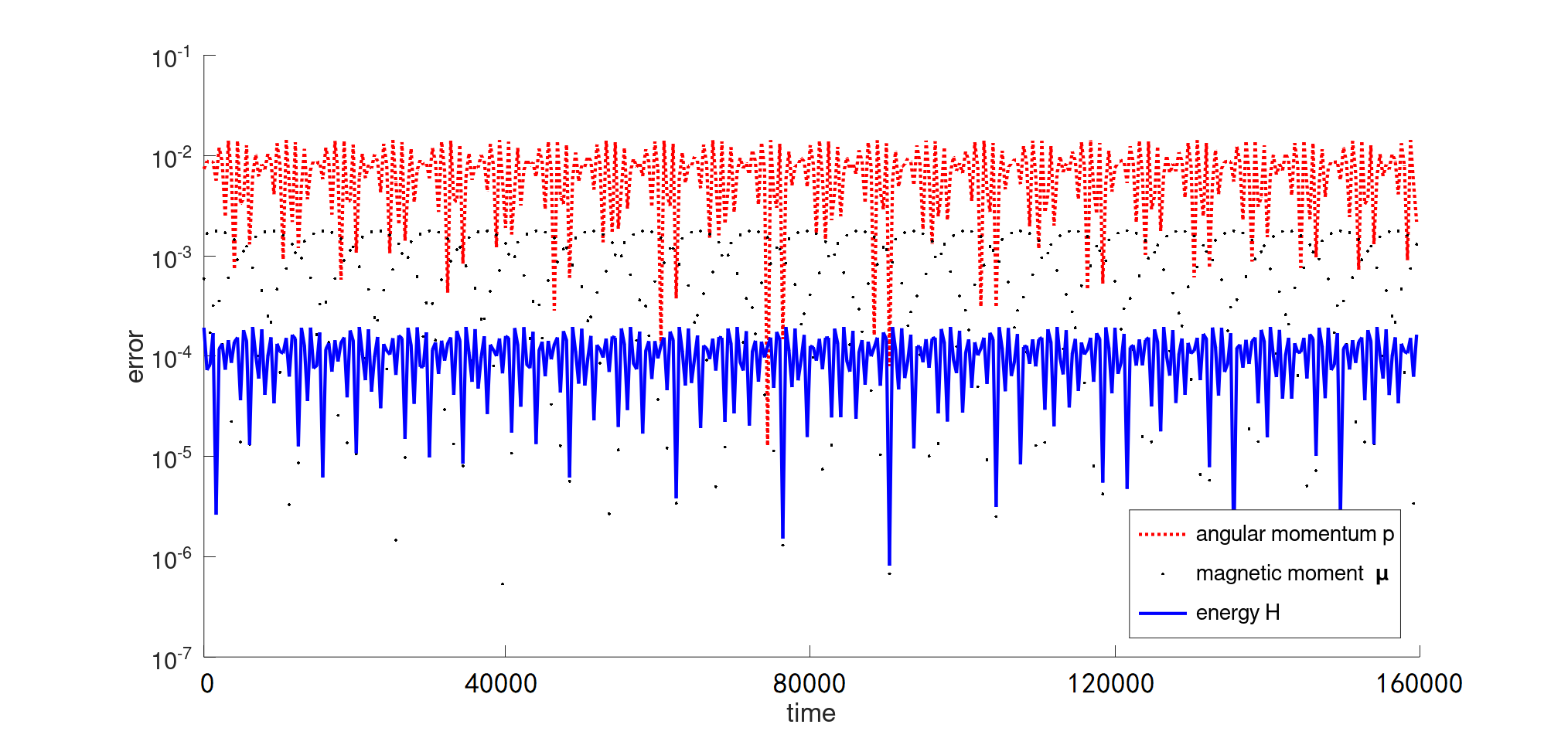}\\
\caption{The Boris method is applied to the simple 2D dynamics with step $h=\pi/10$. (a) The orbit in the 100-st turn; (b) Errors of the angular momentum $p_{\xi}$, the magnetic moment $\mu$ and the energy $H$ for $t\in [0,5\times 10^{5}h]$.}
\label{fig:1}
  (a)\includegraphics[width=0.35\textwidth]{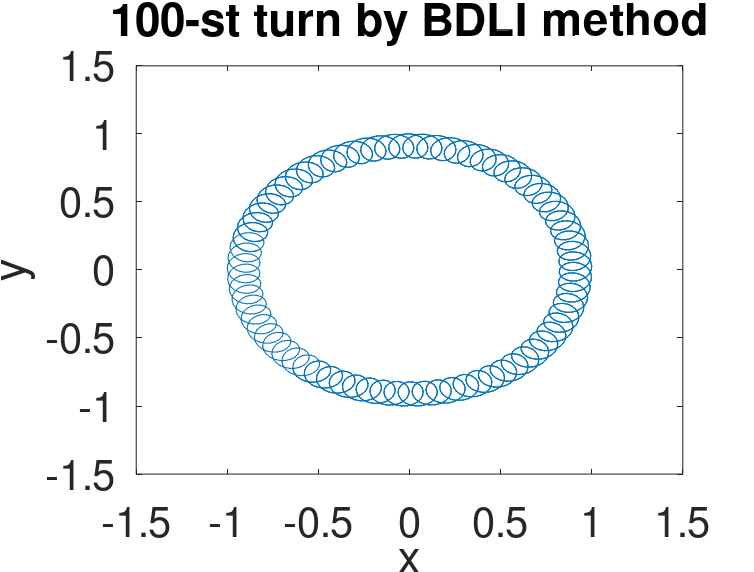}\quad
  (b)\includegraphics[width=0.55\textwidth]{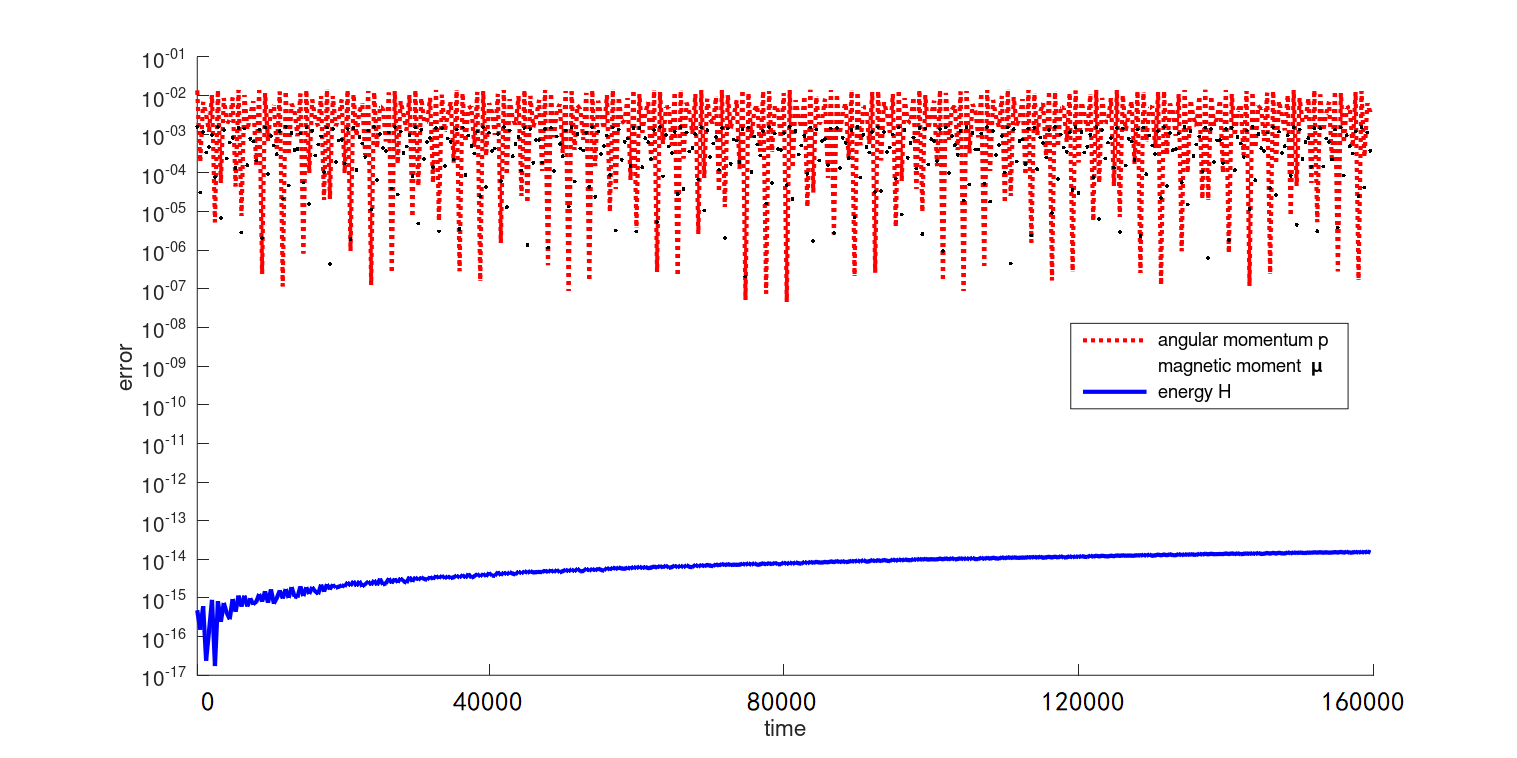}\\
\caption{The BDLI method is applied to the simple 2D dynamics with step $h=\pi/10$. (a) The orbit in the 100-st turn; (b) Errors of the angular momentum $p_{\xi}$, the magnetic moment $\mu$ and the energy $H$ for $t\in [0,5\times 10^{5}h]$.}
\label{fig:2}
\end{figure}

We apply the Boris methods, the BDLI method, the LIM(2,2) method and the CIDG-C method with the same step and compare the results through numerical simulation.
Fig. \ref{fig:1}(a) shows the numerical orbits of the Boris method in the 100-st turn.
Fig. \ref{fig:1}(b) shows the error of the angular momentum $p_{\xi}$, the magnetic moment $\mu$ and the energy $H$ for $t\in [0,5\times 10^{5}h]$.
The numerical results demonstrate that the errors of the three invariants are bounded for a long integration time.

Fig. \ref{fig:2}(a) and Fig. \ref{fig:3}(a) shows the numerical orbits of the BDLI method and the LIM(2,2) method in the 100-st turn.
Fig. \ref{fig:2}(b) and Fig. \ref{fig:3}(b) shows the error of the invariants of this two methods for $t\in [0,5\times 10^{5}h]$.
The BDLI method use Boole's rule to calculate the integrand numerically and the LIM(2,2) method use Guass-Legendre formulas with two points. They are symmetrical and have order 2. The energy can be conserved approximatively and the error of the energy will become smaller until round-off error if we use higher order numerical integration formula.

\begin{figure}[htbp]
  (a)\includegraphics[width=0.35\textwidth]{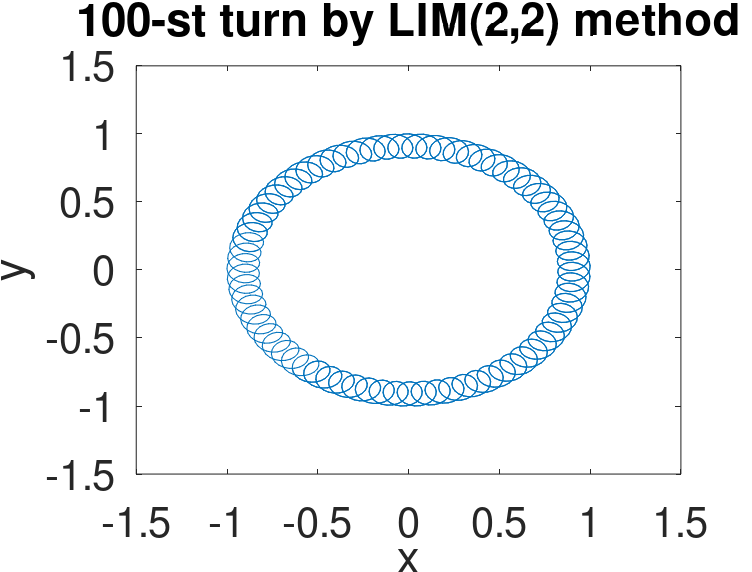}\quad
  (b)\includegraphics[width=0.55\textwidth]{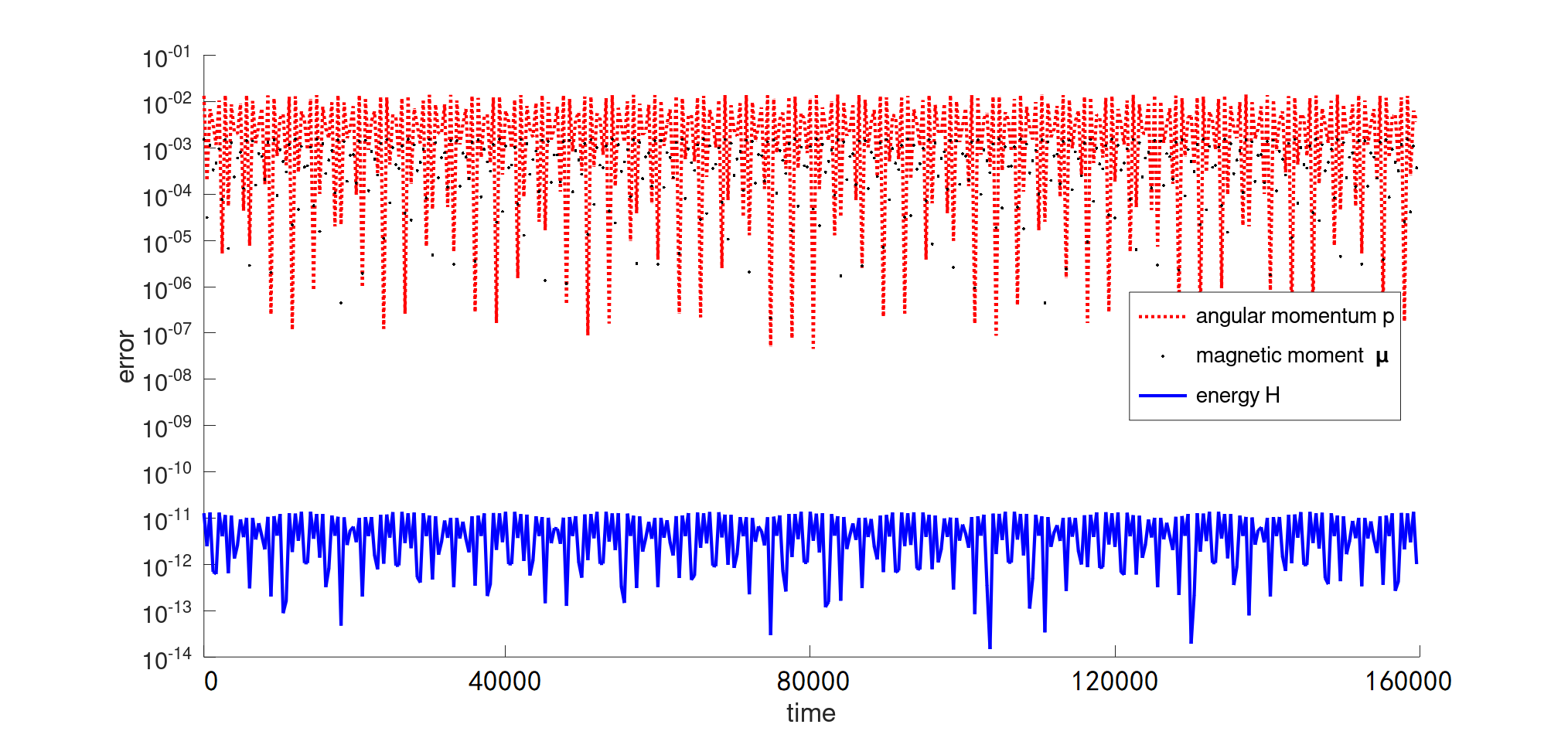}\\
\caption{The LIM(2,2) method is applied to the simple 2D dynamics with step $h=\pi/10$. (a) The orbit in the 100-st turn; (b) Errors of the angular momentum $p_{\xi}$, the magnetic moment $\mu$ and the energy $H$ for $t\in [0,5\times 10^{5}h]$.}
\label{fig:3}
(a)\includegraphics[width=0.35\textwidth]{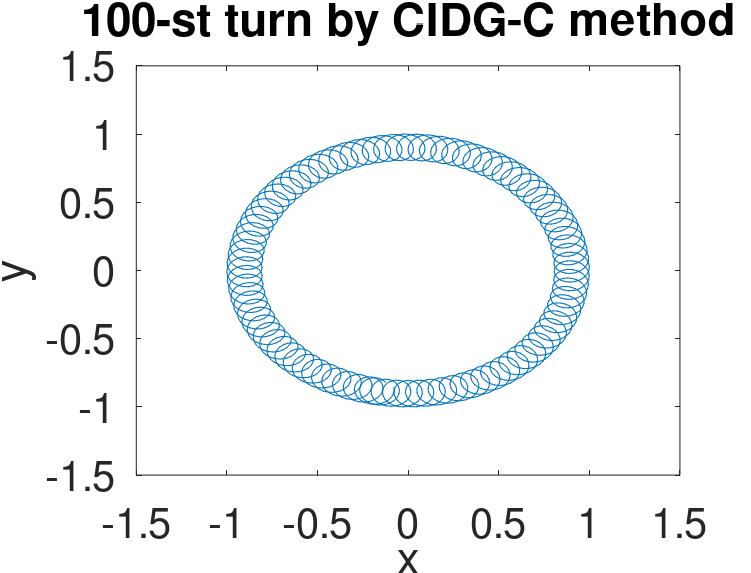}\quad
  (b)\includegraphics[width=0.55\textwidth]{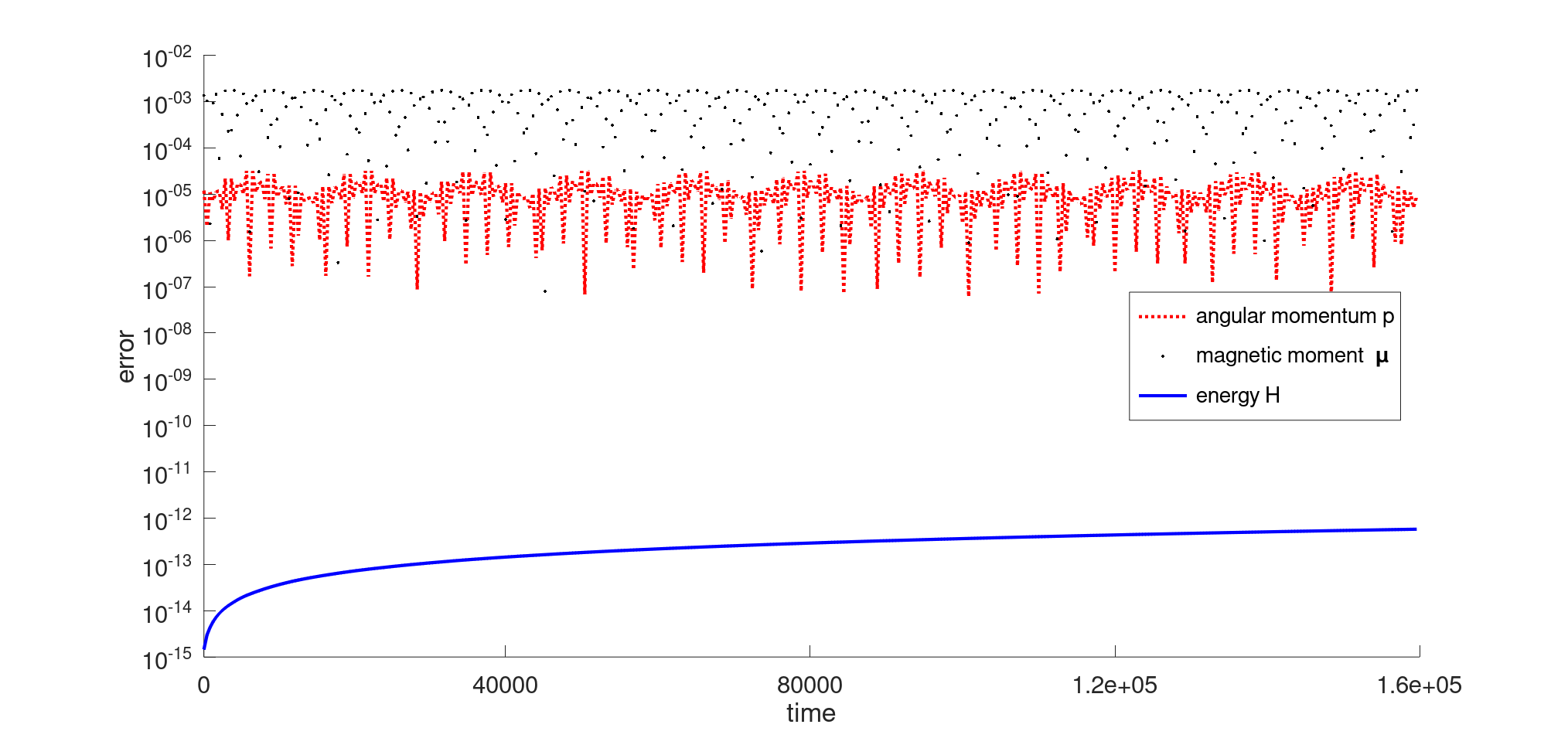}\\
\caption{The CIDG-C method is applied to the simple 2D dynamics with step $h=\pi/10$. (a) The orbit in the 100-st turn; (b) Errors of the angular momentum $p_{\xi}$, the magnetic moment $\mu$ and the energy $H$ for $t\in [0,5\times 10^{5}h]$.}
\label{fig:4}
\end{figure}

Fig. \ref{fig:4}(a) shows the numerical orbits of the CIDG-C method in the 100-st turn.
Fig. \ref{fig:4}(b) shows the error of the invariants for $t\in [0,5\times 10^{5}h]$.
From the figure we can know that the error of the angular momentum $p_{\xi}$ and the magnetic moment $\mu$ are also bounded for a long integration time.
What is more, the CIDG-C method can exactly conserve the energy of the system.

And we compare the CPU times of the four methods from $t = 0$ to $t = 2\times 10^{4}h$ with the same step $h=\pi/10$. The Boris method takes the least amount of time. The Boris method is essentially explicit but it does not preserve the energy as well as the other methods. The LIM(2,2) method is 1.56 times faster than the CIDG-C method and 1.89 times faster than the BDLI method. The BDLI method use Boole’s rule
to calculate the integrand numerically with five points and the LIM(2,2) method use Guass-Legendre
formulas with two points. The BDLI method use more points than the LIM(2,2) method, so the BDLI method takes more time.Table \ref{tab:example}
\begin{table}[H]
	\caption{CPU times of the four methods from $t = 0$ to $t = 2\times 10^{4}h$, with the same 	step $h=\pi/10$.}
	\centering
	\label{tab:example}
	\begin{tabular}{cccc}
	\toprule
	Boris method&BDLI method&LIM(2,2) method&CIDG-C method \\
	\midrule
	3.047 s&91.656 s&48.438 s&75.531 s \\
	\bottomrule
\end{tabular}
\end{table}

\subsection{\textbf{Energy behavior test}}
\label{sec:5s2}

Secondly, we consider the system \eqref{eq:3s2}-\eqref{eq:3s4} with
\begin{align}\label{eq:6s1}
&U(\textbf{x}) = x^{3} - y^{3} + \frac{1}{5}x^{4} + y^{4} + z^{4},
\\\label{eq:6s2}
&\textbf{B}(\textbf{x}) = (0,0,\sqrt{x^{2}+y^{2}})^{T},
\\\label{eq:6s3}
&\textbf{x}(0) = (0,1,0.1)^{T}, \textbf{v}(0) = (0.09,0.55,0.3)^{T},
\end{align}
for which in~\cite{Ref2018Hairer,Ref2019Brugnano} it has been proven that the Boris method exhibits a $O(th^{2})$ drift in the energy.
Fig. \ref{fig:5}(a) shows the energy drift of the Boris method with the step size $h = 10^{-2}$ over the interval $[0, 3\times 10^{4}]$.
The CIDG-C method can conserve the energy with the same step size and there is not energy drift, which is shown in Fig. \ref{fig:5}(b). The LIM(2,2) method and the BDLI method can also conserve the energy with the same step size, which is shown in Fig. \ref{fig:5}(c) and Fig. \ref{fig:5}(d).

\begin{figure}[htbp]
  (a)\includegraphics[width=0.45\textwidth]{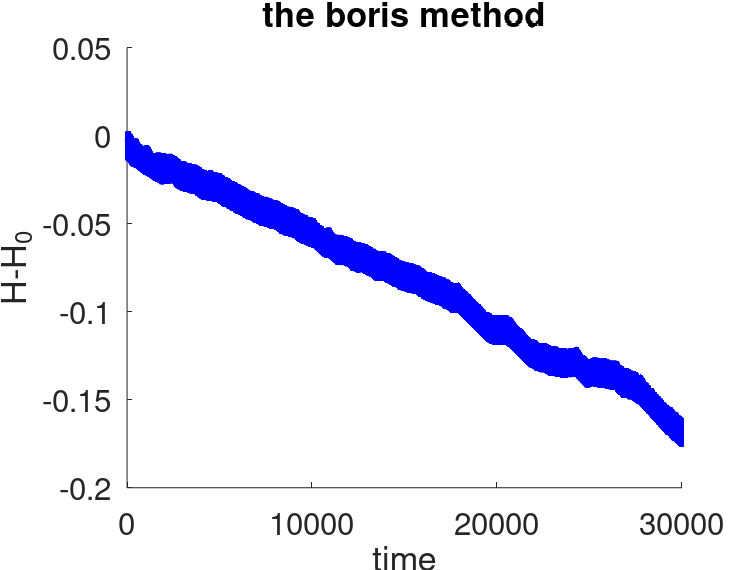}\quad
(b)\includegraphics[width=0.45\textwidth]{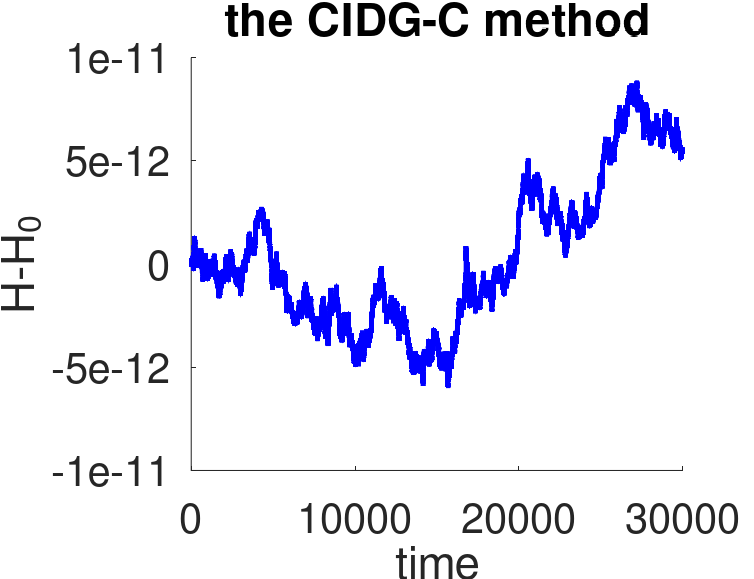}\\
(c)\includegraphics[width=0.45\textwidth]{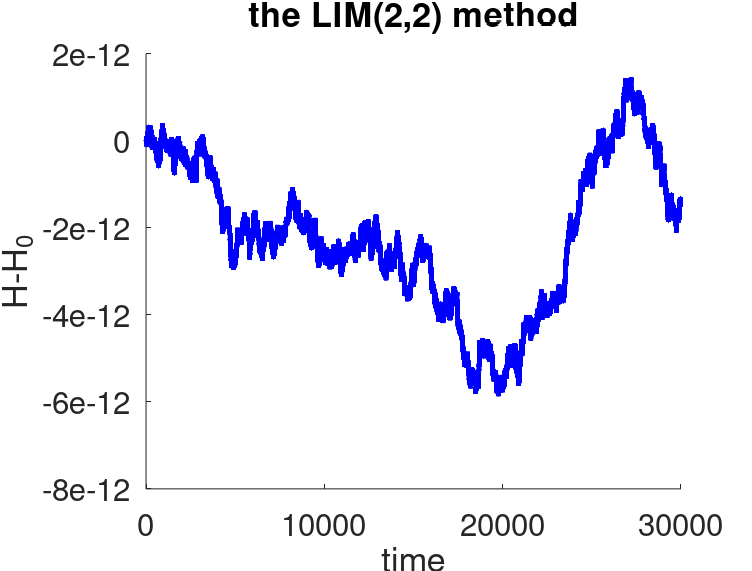}\quad
  (d)\includegraphics[width=0.45\textwidth]{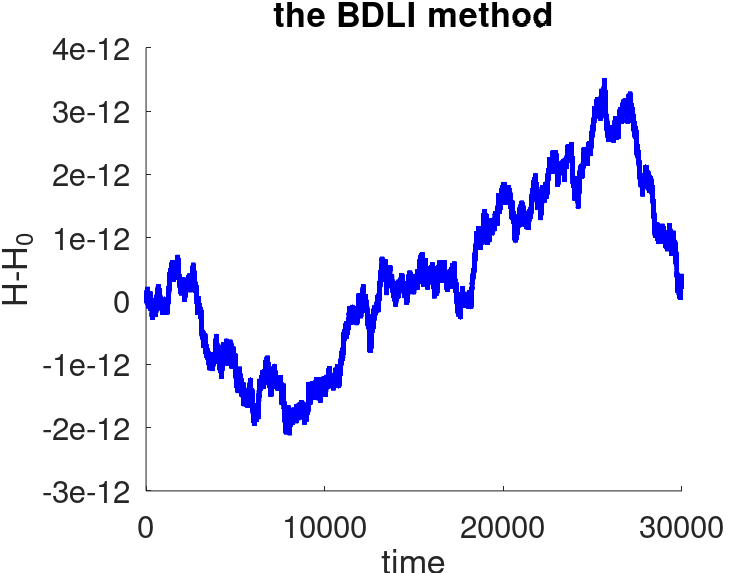}\\
\caption{Hamiltonian error solving problem\eqref{eq:6s1}-\eqref{eq:6s3} with step $h = 10^{-2}$. (a) The Boris method; \textcolor{red}{(b)The CIDG-C method} ;(c)The LIM(2,2) method ;(d)The BDLI method.}
\label{fig:5}
\end{figure}

\begin{remark}\label{remark:3}
\textcolor{red}{Errata} to A novel directly energy-preserving method for charged particle dynamics~\cite{Ref2024Li}: The \textcolor{red}{Fig. 5(b)} in the original paper contains an error, as the magnitude of the error on the vertical axis is incorrect. This is because we accidentally uploaded the wrong image when submitting the final LaTeX file to the editor. The correct image is included as \textcolor{red}{Fig. \ref{fig:5}(b)} in this article. Actually, we are more concerned that the CIDG-C method did not has an energy drift in this example, which is why the error in the \textcolor{red}{Fig. 5(b)} was not detected earlier. We apologize for our mistake.

\end{remark}

\subsection{\textbf{2D dynamics in an axisymmetric tokamak geometry}}
\label{sec:5s3}

Finally, we consider the motion of a charged particle in 2-dimensional axisymmetric tokamak geometry without inductive electric field,
which is also an important application of the guiding center dynamics \cite{Ref2013Qin,Ref2015He}.

The magnetic field in the toroidal coordinates $(r, \theta, \xi)$ is expressed as
\begin{equation}\label{eq:5s3}
\textbf{B}=\frac{B_{0}r}{qR}\textbf{e}_{\theta}+\frac{B_{0}R_{0}}{R}\textbf{e}_{\xi},
\end{equation}
where $B_{0}=1$, $R_{0}=1$, and $q =2$ are constant with their usual meanings. The corresponding vector potential $\textbf{A}$ is chosen to be
\begin{equation}\label{eq:5s4}
\textbf{A}=\frac{z}{2R}\textbf{e}_{R}+\frac{(1-R)^{2}+z^{2}}{4R}\textbf{e}_{\xi}+\frac{\ln R}{2}\textbf{e}_{z}.
\end{equation}

In this example, the physical quantities are normalized as usual. The energy of the system is $H=\frac{1}{2}\textbf{v}^{T} \textbf{v}$. With the conserved quantities, the solution orbit projected on
$(R,z)$ space forms a closed orbit.
In order to use the CIDG-C method, $\textbf{B}$ \eqref{eq:5s3} is transformed to the Cartesian coordinates $(x, y, z)$ which yields
\begin{equation}\label{eq:5s5}
\textbf{B}=-\frac{2y+xz}{2R^{2}}\textbf{e}_{x}+\frac{2x-yz}{2R^{2}}\textbf{e}_{y}+\frac{R-1}{2R}\textbf{e}_{z}.
\end{equation}

\begin{figure}[htbp]
  (a)\includegraphics[width=0.45\textwidth]{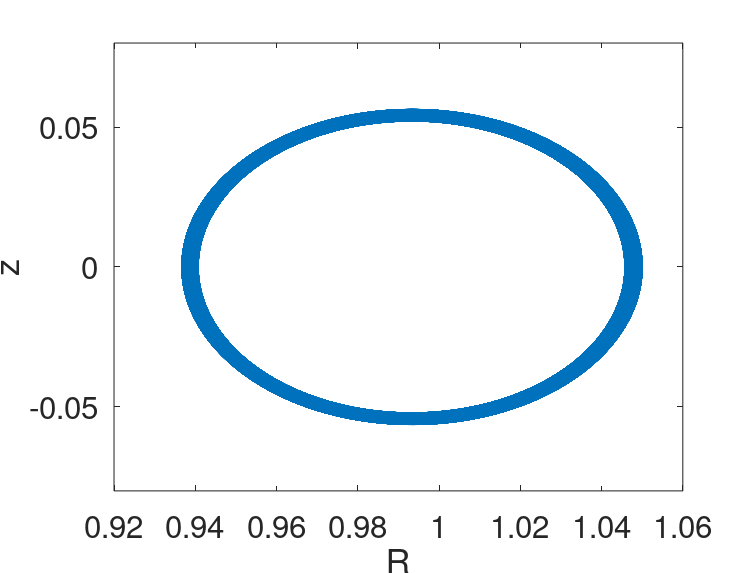}\quad
  (b)\includegraphics[width=0.45\textwidth]{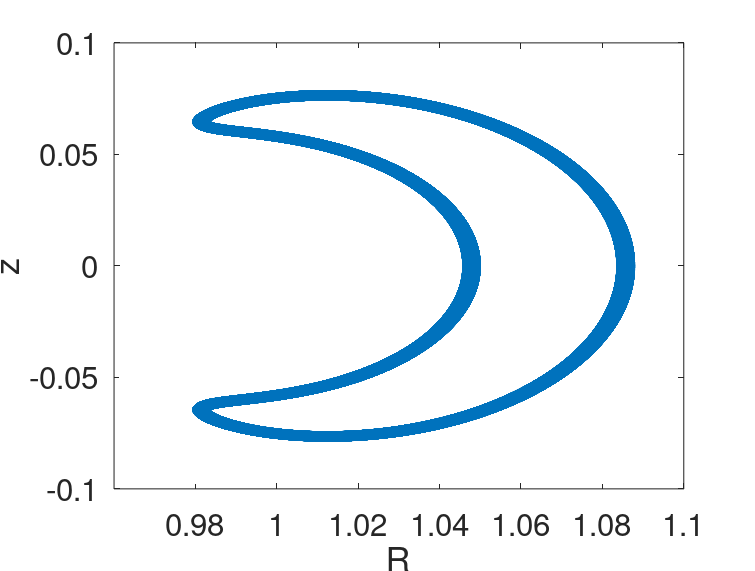}\\
  (c)\includegraphics[width=0.45\textwidth]{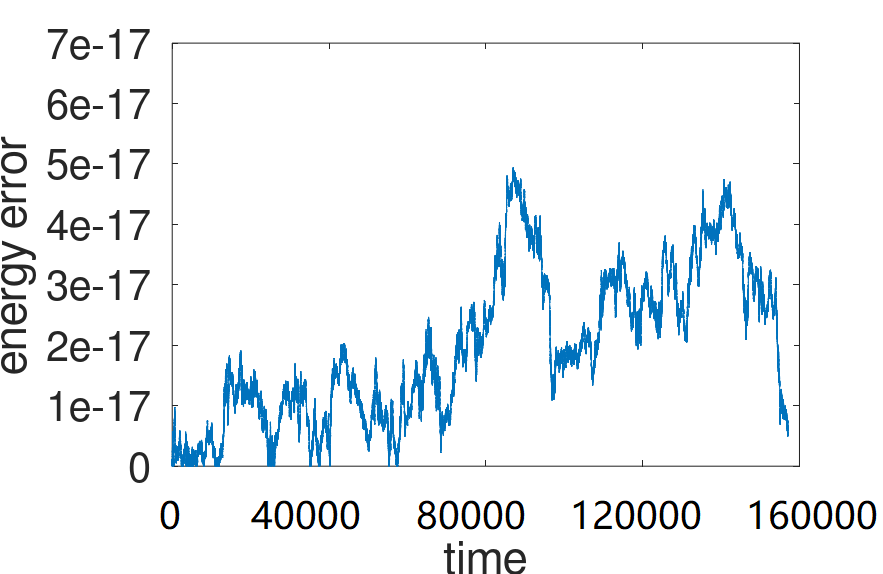}\\
\caption{Numerical solution of the CIDG-C method with step $h=\pi/10$ for $t\in [0,5\times 10^{5}h]$. (a) Transit orbit; (b) Banana orbit; (c) Energy preservation.}
\label{fig:6}
\end{figure}

We firstly consider the initial position $\textbf{x}_{0}=(1.05, 0, 0)^{T}$ and the initial velocity $\textbf{v}_{0}=(0, 2\times 4.816\times 10^{-4}, 2.059\times 10^{-3})^{T}$, the orbit projected on $(R, z)$ space is a transit orbit, and it will turn to a banana orbit when the initial velocity is changed to $\textbf{v}_{0}=(0, 4.816\times 10^{-4}, 2.059\times 10^{-3})^{T}$. We apply the CIDG-C method and use the step size $h =\pi/10$ which is the $1/20$ of the gyro-period $2\pi$.
The simulation over $5\times 10^{5}$ steps is shown in Fig. \ref{fig:6}. Fig. \ref{fig:6}(a) shows the transit orbit solution, Fig. \ref{fig:6}(b) shows the banana orbit solution and Fig. \ref{fig:6}(c) shows the energy preservation.
The CIDG-C method gives correct orbits for a very long integration time, and also maintain the energy precisely $H$ .

\section{Conclusions}
\label{sec:6}

The motion of single-particle, which satisfies the Lorentz force system, has a major role in plasmas.
The Hamiltonian energy is one of the most remarkable features to characterize the system.
In this paper, we apply the coordinate increment discrete gradient method to
the Hamiltonian system, then we derive the CIDG-I method. The new CIDG-I method and its adjoint method are both energy-preserving methods,
and can be combined into a new CIDG-C method.
The CIDG-C method is symmetrical and can exactly conserve
the Hamiltonian energy without numerical quadrature formula.
Numerical results have shown that
the CIDG-C method surpasses the Boris method in long time simulation.

\vspace{0.3cm}
\hspace{-0.5cm}{\bf Acknowledgements}\\
This work is supported by the National Natural Science Foundation of China (No. 12101072) and the Natural Science Foundation of Chongqing (CSTB2022NSCQ-MSX0890).

\vspace{0.5cm}
\hspace{-0.5cm}

\end{document}